\documentclass[a4paper,11pt]{amsart}
\usepackage{amssymb}
\usepackage{amsmath}

\usepackage{color}

\newtheorem{theorem}{Theorem}[section]
\newtheorem{lemma}[theorem]{Lemma}
\newtheorem{proposition}[theorem]{Proposition}

\theoremstyle{definition}
\newtheorem{definition}[theorem]{Definition}
\newtheorem{example}[theorem]{Example}

\theoremstyle{remark}
\newtheorem{remark}[theorem]{Remark}
\numberwithin{equation}{section}

\def\id{{\bf 1}\!\!{\rm I}}

\def\A{{\mathcal A}_t (\mathcal T) }

\def\bn{{\mathbb N}}

\def\br{{\mathbb R}}

\def\A{{\mathcal A}_t (\mathcal T) }

\def\s{\sigma}

\def\l{\lambda}

\def\t{\tau}
\def\f{\varphi}

\def\a{\alpha}

\def\b{\beta}
\def\m{\mu}

\def\d{\delta}
\def\g{\gamma}

\def\cf{{\mathcal F}}
\def\wb{{\mathbf{w}}}

\def\ca{\mathcal{ A}}

\def\cf{\mathcal{ F}}

\def\ck{\mathcal{K}}

\def\ct{\mathcal{ T}}

\def\cw{\mathcal {W}}

\begin{document}
\setcounter{page}{1}

\title[Residualities and uniform ergodicities]{Residualities and uniform ergodicities of Markov semigroups}

\author[Nazife Erkur\c{s}un-\"Ozcan, Farrukh Mukhamedov]{Nazife Erkur\c{s}un-\"Ozcan$^{1,*}$ and Farrukh Mukhamedov$^{2,*}$}

\address{$^{1}$ Department of Mathematics, Faculty of Science, Hacettepe University, Ankara, 06800, Turkey.}
\email{{erkursun.ozcan@hacettepe.edu.tr}}

\address{$^{2}$ Department of Mathematical Sciences, College of Science, United Arab Emirates University, 15551, Al-Ain, UAE.} \email{{far75m@gmail.com; farrukh.m@uaeu.ac.ae}}




\date{Received: xxxxxx; Revised: yyyyyy; Accepted: zzzzzz.
\newline \indent $^{*}$ Corresponding author}

\begin{abstract}
The first goal of the present paper is to study residualities of the set of uniform $P$-ergodic Markov semigroups defined on abstract state spaces by means of a generalized Dobrushin ergodicity coefficient. In the last part of the paper, we explore uniform mean ergodicities of Markov semigroups.
\end{abstract} \maketitle


\section{Introduction}

It is known that Markov operators and semigroups appear in many branches of applies sciences. For instance, Boltzmann equation and its simplified
version Tjon-Wu equation generate a nonlinear Markov semigroups. In possible treatments, it is convenient to investigate Markov operators on densities, sometimes it is more convenient to consider Markov operators on measures \cite{LM94}. Notice that semigroups of Markov operators can be generated by partial differential equations (transport equations) \cite{RPT02}.

Several papers have delved into the exploration of general characteristics associated with Markov operators, as evidenced by works such as \cite{B81, BK, Iw, R97}. Conventionally, the dominant generic characteristic ascribed to Markov operators/semigroups is the notion of asymptotic stability, wherein all elements exhibiting this trait are considered residual \cite{LM92}. A recent development in this context is presented in \cite{Kuna}, which establishes a similar outcome for Markov semigroups operating on the set of density operators within a Hilbert space. In this context, "residuality" refers to a subset of a complete metric space, with its complement capable of being expressed as a countable union of nowhere dense sets.

Recently, an extended Dobrushin ergodicity coefficient $\delta_P(T)$ has been introduced for Markov operators, acting on an abstract state space, in relation to a projection $P$ \cite{MA}. Additionally, the investigation of a generalized Dobrushin ergodicity coefficient $\delta_P(T)$ for Markov $C_0$-semigroups on abstract state spaces with respect to a projection $P$ has been undertaken in \cite{EM-JMS}. In the realm of open quantum systems, \cite{SU22} explores invariant subspaces in the presence of symmetries, which are associated with projections $P$ of the corresponding semigroups (refer also to \cite{CG21}). When the projection $P$ is a rank-one projection, the study of the uniform convergence of irreducible Markov semigroups has been explored in \cite{AFP21,FR1,GM22}. Furthermore, the Dobrushin ergodicity coefficient $\delta(T)$ has been employed to examine the asymptotic stability of Markov $C_0$-semigroups on abstract state spaces in \cite{EM2018, EM2018Q}. Particularly noteworthy are the perturbation results for $C_0$-semigroups of Markov operators obtained in \cite{EM2020,EM-JMS}. It is important to note that within this abstract framework, both classical and quantum cases can be considered as specific instances \cite{Alf,E}.

It is noted that the averages of operators appear in several fields of
mathematics, theoretical physics, and other fields of science \cite{AFP21,BFH20,GOSS22,RPT02}. In \cite{EO1, EO2}, convergence of averages are deeply studied on different spaces. 

In the present paper, we study residualities of the set of uniform $P$-ergodic Markov semigroups defined on abstract state spaces by means of a generalized Dobrushin ergodicity coefficient. Furthermore, we are going to explore the uniformly mean ergodicity of Markov semigroups in terms of the generalized Dobrushin's ergodicity coefficient without irreducibility condition.

To best of our knowledge the obtained results are new even in the classical setting. Namely, if $X$ is taken as either $L^1(E,\m)$ or a predual of a von Neumann algebra, then all the established results shed new light into the classical and non-commutative settings, respectively.

\section{Preliminaries}

In this section, essential definitions and findings pertaining to abstract state spaces are revisited.

 Let $X$ denote an ordered vector space equipped with a generating positive cone $X_+$ (i.e., $X = X_+ - X_+$). Suppose that $X_+$ has a base denoted as $\mathcal{K}$, where 
 $$\mathcal{K} = \{x \in X_+ : f(x) = 1\},$$ 
 for a strictly positive linear functional $f$ on $X$. Define $U := \mathcal{K} \cup (-\mathcal{K})$, and introduce a seminorm as follows:
$$
\|x\|_{\ck}=\inf\{\l\in\br_+:\ x\in\l U\}.
$$
In this context, one finds that $\mathcal{K} = \{x \in X_+ : \|x\|_{\mathcal{K}} = 1\}$, and $f(x) = \|x\|_{\mathcal{K}}$ for $x \in X_+$. Assuming the seminorm becomes a norm and $X$ is a complete space with respect to $\|\cdot\|_{\mathcal{K}}$, and $X_+$ is closed, the tuple $(X, X_+, \mathcal{K}, f)$ is termed an \textit{abstract state space}. In this scenario, $\mathcal{K}$ is identified as a closed face of the unit ball of $X$, and $U$ encompasses the open unit ball of $X$. It's noteworthy that if $U$ is \textit{radially compact} (as defined in \cite{Alf}), meaning that the intersection of $\ell$ and $U$ forms a closed and bounded segment for every line $\ell$ passing through the origin of $X$, then $\|\cdot\|_{\mathcal{K}}$ becomes a norm. The radial compactness condition is equivalent to $U$ coinciding with the closed unit ball of $X$. In the latter case, $X$ is called as a \textit{strong abstract state space}. Throughout, we employ $\|\cdot\|$ in place of $\|\cdot\|_{\mathcal{K} }$. For a deeper comprehension of the distinction between a strong abstract state space and a broader class of base norm spaces, the reader is directed to \cite{JP21, Yo}.

It is worth noting that fundamental examples of abstract state spaces encompass classical $L_1$-spaces and the space of density operators operating on certain Hilbert spaces (see, for details \cite{Alf, WN}).

In the sequel, for the sake of simplicity, always $X$ stands for an abstract state space. 
%
%
%

In the current paper, we are going to consider general abstract state spaces for
which the convex hull of the base $\ck$ and $-\ck$ is not supposed to be radially compact.

A linear operator $T:X\to X$ is
called \textit{positive}, if $Tx\geq 0$ whenever $x\geq 0$. If a positive linear operator $T: X \to X$ satisfies $T(\mathcal{K}) \subset \mathcal{K}$, it is referred to as a "Markov" operator. It is evident that $\|T\| = 1$, and its adjoint mapping $T^*: X^* \to X^*$ operates in the ordered Banach space $X^*$ with the unit being $f$. Furthermore, it holds that $T^*f = f$.



The following concept represents a generalized form of Dobrushin's ergodicity coefficient (see \cite{MA}), defined as follows: .

\begin{definition}
Given a linear bounded operator $T:X\to X$ one defines
\begin{equation}
\label{db} \d_P (T)=\sup_{x\in N_P,\ x\neq 0}\frac{\|Tx\|}{\|x\|}
\end{equation}
where \begin{equation}
\label{NN}  N_P=\{x\in X: \ P(x)=0\}.
\end{equation}
If $P=I$, we put $\d_P (T)=1$. The quantity $\d_P (T)$ is called the \textit{generalized Dobrushin ergodicity coefficient} of $T$ with respect to $P$.
\end{definition}

\begin{remark}
Assume that $P$ is a one-dimensional projection, i.e. $Px=f(x)y_0$ ($y_0 \in \ck$), then the Dobrushin ergodicity coefficient $\delta_P(T)$ for $T$ is equal to the classical Dobrushin ergodicity coefficient $\delta(T)$ studied in \cite{C, D, M01}. It's noteworthy that in the case where $X^*$ evolves into a von Neumann algebra, a similar coefficient was introduced in \cite{M}.
\end{remark}

Let us use the notation $\Sigma (X)$ for the set of all Markov operators on $X$ and $\Sigma_P (X)$ for the set of all Markov operators $T$ on $X$ with $PT =TP$.

The next result establishes several properties of the generalized Dobrushin
ergodicity coefficient.

\begin{theorem}\label{Dob}\cite{MA}
Let $X$ be an abstract state space, $P$ be a projection on $X$ and let  $T,S  \in \Sigma (X)$. Then the following
statements are true:
\begin{itemize}
\item[(i)] $0 \leq \d_P (T) \leq 1$;
\item[(ii)] $|\d_P (T) - \d_P (S)| \leq \d_P (T-S) \leq \left\| T - S\right\|$;
\item[(iii)] if $H: X\to X$ is a bounded linear
operator with  $HP = PH$, then
$$
\d_P (TH) \leq \d_P (T) \left\| H \right\| ;
$$
\item[(iv)] if $H: X\to X$ is a bounded  linear
operator with $PH =0$, then
$$
\left\| TH\right\| \leq \d_P (T) \left\| H \right\| ;
$$
\item[(v)] if $S \in \Sigma_P (X)$, then
$$
\d_P (TS) \leq \d_P (T) \d_P (S).
$$
\end{itemize}
\end{theorem}

\begin{remark}\label{SS1}
If $X$ is a strong abstract state space, then due to \cite[Proposition 3.9]{MA} one has
\begin{equation}\label{SS11}
\d_P (T) = \frac{1}{2} \sup \{\left\| Tu -
Tv\right\| : u,v \in \ck u-v \in N_P\}.
\end{equation}
\end{remark}

The main object in this paper is a \textit{$C_0$-semigroup} on a Banach space $ X $, i.e.  a map $ {\displaystyle T_t :\mathbb {R} _{+}\to L(X)}$ such that
\begin{itemize}
\item[(i)] ${\displaystyle T_0=I} $,   identity operator on ${\displaystyle X}$;
\item[(ii)] for all $t,s\geq 0$ one has $T_{t+s}=T_tT_s $;
\item[(iii)] for every $x_{0}\in X$ one has $\|T_t x_{0}-x_{0}\|\to 0$, as ${\displaystyle t\downarrow 0} $.
\end{itemize}


A $C_0$-semigroup $\mathcal T=(T_t)_{t\geq 0}$ is called \textit{$C_0$-Markov semigroup} if for every $t\in\br_+$, the operator $T_t$ is a Markov operator. Denote
$$
\Xi (X ) = \{ \mathcal T = (T(t))_{t\geq 0} : \mathcal T \  \text{is a} \ C_0-\text{Markov semigroup}\}.
$$

An
element $x_0\in X$ is called a \textit{fixed point}  or \textit{stationary point} of  $\mathcal
T=(T_t)_{t\geq 0}$ if  $T_tx_0=x_0$ for all $t\in\br_+$. The basic structures of positive semigroups defined on
ordered Banach spaces are presented in \cite{BR}.


Let $P$ be a projection of $X$ and $\mathcal T=(T_t)_{t\geq 0}$ be a Markov semigroup.
We write $\mathcal T\in\Xi_P (X)$ if for each $t\in\br_+$ one has $T_t\in \Sigma_P(X)$. This is equivalent to $T_tP=PT_t$ for all $t\geq 0$.


\section{Residuality of Uniform-$P$-ergodic Markov semigroups}

In \cite{EM-JMS}, uniformly $P$-ergodicity of $C_0$-Markov semigroups is studied. 

\begin{definition} A $C_0$-Markov semigroup  $\mathcal T=(T_t)_{t\geq 0}$ defined on $X$ is called
\begin{enumerate}
\item[(i)] {\it uniformly $P$-ergodic} if there exists  $P$ such that
$$
\lim_{t\to\infty}\|T_t - P \|=0;
$$
\item[(ii)] if $P$ is a one-dimensional projection, then uniformly $P$-ergodic semigroup is called  {\it uniformly ergodic}.
\end{enumerate}
\end{definition}

\begin{theorem}\label{Dobb} \cite{EM-JMS}
Let $X$ be an abstract state space. Assume that $P$ is a Markov projection  and $\mathcal T =(T_t)_{t\geq 0}$
is a $C_0$-Markov semigroup on $X$. Then the following statements are equivalent:
\begin{itemize}
  \item[(i)] $\mathcal T$ is uniformly $P$-ergodic;
  \item[(ii)] There is a strongly-continuous semigroup $(Q_t)_{t\geq 0}$ with $Q_0=I-P$ on $X$ such that $T_t=P+Q_t$ and $PQ_t=Q_tP=0$ (for all $t\geq 0$). There is $t_0>0$ such that $\|Q_{t_0}\|<1$. Moreover,
\begin{equation}\label{Qt}
\d_P(T_{t})\leq \|Q_{t}\|\leq 2\d_P(T_{t}), \ \ \textrm{for all} \ \ t\geq 0.
\end{equation}
  \item[(iii)] There exists $t_0$ such that $\d_P (T_{t_0}) < 1$ and $PT_t=T_tP=P$ for all $t\geq 0$. Moreover,
one finds positive constants $C, \a, t_0 \in\br$ with
\begin{equation}\label{Dobrushin}
\left\| T_t - P \right\| \leq C e^{-\a t}, \,\,\,\,\, \forall
t\geq t_0.
\end{equation}
\end{itemize}
\end{theorem}

Recall that for a bounded linear operator $T$ on $X$ one denoted by $\s(T)$ its spectrum. Then the spectral radius of $T$ is defined to be $r(T)= \sup_{\lambda\in\sigma (T)} |\lambda |$. Due to \cite[p. 687]{EM21}, we have the next fact.

\begin{lemma} \label{LS}
Let $P$ be a Markov projection on $X$. Assume that $\mathcal T =(T_t)$ is uniformly $P$-ergodic $C_0$-Markov semigroup on $X$.
Then
$$
\lim_{n\to\infty}(\d_P (T_n))^{1/n} = \lim_{n\to\infty}(\d_P ((T_1)^n))^{1/n}= r(T_1-P).
$$
\end{lemma}
%

\begin{theorem}
Let $P$ be a Markov projection on $X$. Assume that $\mathcal T =(T_t)$ is a uniformly $P$-ergodic $C_0$-Markov semigroup on $X$.
Then $\d_P (T_t)= e^{-\alpha t}$ for all $t > 0$ if and only if $\d_P (T_t) = r(T_t- P)$ for all $t > 0$.
 \end{theorem}

 \begin{proof}
 Since $\mathcal T$ is a $C_0$-Markov semigroup which is commutative, then by \cite[Theorem 1.6.17]{Richart} , $\sigma (T_t)$ is a continuous function of $t$. Now, let us assume that $\d_P (T_t)= e^{-\alpha t}$ for all $t > 0$. Then by Lemma \ref{LS}
$$
r(T_1 - P)= \lim_{n\to\infty}\d_P (T_1^n)^{1/n} = \lim_{n\to\infty} (e^{-\alpha n})^{1/n} = e^{-\alpha}.
$$
Due to $T_n - P= T_1^n  - P = ( T_1 - P)^n$, it follows that
$r (T_n - P)= r((T_1 - P)^n)$ and
$r (T_{\frac{n}{m}} - P) = r((T_1 - P)^{\frac{n}{m}})$. The continuity of $\sigma(T_t)$ implies
$$
r (T_t - P)= r((T_1 - P)^t) = e^{-\alpha t} = \d_P (T_t), \  \  \  \forall t > 0.
$$

Now, suppose that
$\d_P (T_t)= r(T_t - P)$. Then
$$
r(T_t - P)= (r(T_1 - P))^t = e^{\log (r(T_1 - P)) t}
$$
where $\alpha = - \log (r(T_1 - P))$. This completes the proof.
 \end{proof}

Furthermore, we are going to study residuality results for Markov semigroups.

By $C ([0,\infty), L (X))$ we denote the set of all continuous maps $\mathbf{F}:[0,\infty)\to L(X)$, where $L(X)$ stands for the Banach space of all bounded linear operators on $X$. This set is equipped with the metric
$$
\rho(\mathbf{F},\mathbf{G})=\sum_{m=1}^\infty\frac{1}{2^m}\frac{\rho_m(\mathbf{F},\mathbf{G})}{1+\rho_m(\mathbf{F},\mathbf{G})},
$$
where
\begin{equation}\label{rho}
\rho_m (\mathbf{F},\mathbf{G}) = \sup_{t\in [0,m]} \|  F_t - G_t \|.
\end{equation}

It is well-known \cite{Sik} that $(C ([0,\infty), L (X)), \rho)$ is a complete metric space.


\begin{lemma}\label{X}
Let $X$ be an abstract state space and $\mathcal T, \mathcal S\in \Xi(X)$.  For each $r > 0$, we define
\begin{equation}\label{rho}
\rho_r (\mathcal T, \mathcal S) = \sup_{t\in [0,r]} \|  T_t - S_t \|.
\end{equation}
Then $\rho_r$ is a metric on $\Xi(X)$. Moreover, all $\rho_r$ and $\rho$ are equivalent on $\Xi (X)$.
\end{lemma}

\begin{proof}
The proof is the same as in \cite{Kuna}.
\end{proof}

\begin{lemma}\label{X1}
$\Xi(X)$ is a $\rho_r$-closed subset of $C ([0,\infty), L (X))$.
\end{lemma}

\begin{proof}
Let ${\mathcal T}_n \xrightarrow{\rho_r} \mathcal H $, when ${\mathcal T}_n \in \Xi(X)$. From \eqref{rho} one gets
$$
\| {T}_{n,t}- H_t\| \leq \rho_r (\mathcal T, \mathcal P),  \   \  \forall t \in [0,r].
$$
Hence, by Lemma \ref{X} we infer that  $T_{n,t}- H_t\to 0 $ as $n\to\infty$, for all $t\geq 0$.

Let us establish that for each $t\geq 0$ the operator $P_t$ is Markov. Indeed, let $x\in \ck$, then
\begin{eqnarray*}
|1 - f( H_tx)| &=& | f( T_{n,t} x ) - f(H_tx) | \\
&\leq& \| (T_{n,t} - H_t) x\|\\
& \leq& \| T_{n,t}- H_t\| \| x\| \to 0.
\end{eqnarray*}
It implies that $f(H_tx) = 1$. The positivity of $T_{n,t}$ implies $H_t$ is also positive.

Now, we show that $H_t$ is a $C_0$-semigroup. Indeed, we have
\begin{itemize}
\item[(i)] $\| H_0 - I \| = \| H_0- T_0+ T_0- I \| = \| H_0- T_0 \| \to 0 \  \  as \  n\to \infty$
which implies $H_0=I$

\item[(ii)]
\begin{eqnarray*}
\| H_{t + s} - H_tH_s \|&\leq& \|H_{t + s} - T_{t + s}\|+\|T_{t + s}-T_tT_s\|+
\| T_tT_s- H_tH_s\|  \\
&\leq& \|H_{t + s} - T_{t + s}\|+\| T_t\|\|T_s- H_s\|+\|H_s\|\|T_t-H_t\|\to 0
\end{eqnarray*}
Hence $H_{t + s}=H_tH_s$.

\item[(iii)] $
\displaystyle
\lim_{t\to 0} \|H_t x -x \| \leq \lim_{t\to 0} (\|H_tx - T_{n,t} x \| + \| T_{n,t}x - x \|) \to 0, $  because ${\mathcal T}_n \xrightarrow{\rho_r} \mathcal H $, and $\mathcal H$ is a $C_0$-semigroup.
\end{itemize}
This completes the proof.
\end{proof}

%

\begin{lemma}
 Let $\mathcal T \in \Xi_P(X)$. Then the following conditions are equivalent:
 \begin{itemize}
 \item[(i)] The semigroup $\mathcal T$ is uniformly $P$-ergodic;

 \item[(ii)] There exists a Markov projection $P$ with $T_tP=P$ such that for some (all) $t_0 > 0$ we have $\lim\limits_{n\to \infty} \|T_{nt_0} - P\| =0$.

 \item[(iii)] There exists a Markov projection $P$ with $T_tP=P$ and $t_n\to\infty$ such that $\lim\limits_{t_n\to \infty} \|T_{t_n} - P\| =0$.
 \end{itemize}
\end{lemma}

\begin{proof} The implications  $(i) \Rightarrow (ii) \Rightarrow (iii)$ are obvious. Therefore, it is enough to establish $(iii) \Rightarrow (i)$.

Assume that (i) fails, then for some $\epsilon >0$ and $s_j \to \infty$, we have
$$
\sup_{x\in \ck} \|T_{s_j} x - Px \| \geq \epsilon.
$$

Let us choose $n_j \to \infty$ such that $s_j > t_{n_j}$ and denote $r_j = s_j - t_{n_j} \geq 0$.
\begin{eqnarray*}
\displaystyle
\sup_{x\in \ck} \|T_{s_j} x - Px \| &=& \sup_{x\in \ck} \| T_{t_{n_j}}T_{r_j}  x - PT_{r_j}x \|  \\
&=& \sup_{y\in T_{r_j} (\ck)} \|T_{t_{n_j}} y - Py \| \\
&\leq& \sup_{x\in \ck} \|T_{t_{n_j}} x - Px \|  \to 0
\end{eqnarray*}
which is a contradiction.
\end{proof}

Denote
\begin{eqnarray*}
& & \Xi_P^{u} (X) := \{\mathcal T \in \Xi_P (X): \mathcal T \  \text{is uniformly $P$-ergodic}\}, \\
& & \Xi_P^{inv} (X) := \{\mathcal T \in \Xi_P  (X): \ T_tP =P, \forall t\geq 0\}.
\end{eqnarray*}

It is obvious that $ \Xi_P^{u} (X)\subset \Xi_P^{inv} (X)$.

\begin{lemma}\label{TPO}
Let $X$ be an abstract state space and $P$ be a Markov projection on $X$. Then the set $\Xi_P^{u} (X)$ is $\rho_1$ open in $\Xi_P^{inv} (X)$.
\end{lemma}

\begin{proof} Take any $\mathcal{T}\in \Xi_P^{u} (X)$. Then, by Theorem \ref{Dobb} there exists $t_0$ such that $\d_P (T_{t_0})=1-\varepsilon$, for some $0<\varepsilon<1$.  Assume that $N\in\mathbb{N}$ such that $\frac{t_0}{N}<1$. Denote
$$
U(\mathcal{T},\varepsilon)=\bigg\{\mathcal{R}\in \Xi_P^{inv} (X): \ \rho_1(\mathcal{R},\mathcal{T})<\frac{\varepsilon}{2N}\bigg\}.
$$

Now, let us establish that $U(\mathcal{T},\varepsilon)\subset \Xi_P^{u} (X)$. Indeed, take any $\mathcal{R}\in U(\mathcal{T},\varepsilon)$, then
\begin{eqnarray*}
\d_P (R_{t_0}) &\leq& \d_P (T_{t_0})+\d_P (R_{t_0}-T_{t_0})\\[2mm]
&\leq & 1-\varepsilon+ \|R_{t_0}-T_{t_0}\|\\[2mm]
&=&1-\varepsilon+ \|(R_{t_0/N})^N-(T_{t_0/N})^N\|\\[2mm]
&\leq &1-\varepsilon+ N\|R_{t_0/N}-T_{t_0/N}\|\\[2mm]
&< &1-\varepsilon+ N\frac{\varepsilon}{2N}\\[2mm]
&=&1-\frac{\varepsilon}{2}
\end{eqnarray*}
which, again by Theorem \ref{Dobb} yields that $\mathcal{R}\in \Xi_P^{u} (X)$.
This completes the proof.
\end{proof}

Let $A$ ba the generator of $T_t$. Then for $Q \in L(X)$, then the operator $C := A + Q$
with $D(C) := D(A)$ generates a $C_0$-semigroup $
(T^Q_t)_{t\geq 0}$
 satisfying
$$
\left\| T^Q_t\right\| \leq  M e^{(w+M \left\|Q\right\|)t} \ \ \forall
t \geq 0.
$$
and
$$
\displaystyle
 T^Q_t = T_t + \int_0^t T_{t - s}QT^Q_sds
$$
holds for every $t \geq 0 $ and $x \in X$.

Moreover, the semigroup $(T^Q_t)_{t\geq 0}$ could be written as follows
$$\displaystyle T^Q_t = T^Q_{0,t} + \sum_{k=1}^{\infty} T^Q_{k,t},
$$
where $T^Q_{0,t}= T_t$ and inductively
\begin{equation}\label{TQ1}
T^Q_{k+1,t} = \int_0^t T_{t-s}QT^Q_{k,s} ds.
\end{equation}
The series is convergent in the uniform operator norm and the convergence is uniform on bounded intervals.

Now, following the Phillips perturbation theorem \cite{DS}, the generator $A + \lambda (Q-I)$ defines a semigroup  $\mathcal{T}_{\lambda,Q}=(T^{\lambda, Q-I}_t)$ by
\begin{equation}\label{TQ2}
T^{\lambda, Q-I}_t = e^{-\lambda t}\bigg(T_t+ \sum_{k=1}^{\infty} \lambda^k T^Q_{k,t}\bigg).
\end{equation}

Now, we are going to establish that if $\mathcal T$ is Markov semigroup and $Q$ is a Markov operator, then the defined semigroup is Markov as well. This type of result appears in diverse areas as applied  sciences \cite{MD,RPT02}.

\begin{lemma}\label{TQM}
Let $X$ be an abstract state space and $Q$ be a Markov operator on $X$. If $\mathcal T \in \Xi(X)$ then $\mathcal{T}_{\lambda,Q}\in \Xi(X)$. Moreover, if $\mathcal T \in \Xi_P(X)$, then $\mathcal{T}_{\lambda,P}\in \Xi_P(X)$, for a Markov projection $P$.
\end{lemma}

\begin{proof}
Since $Q$ and $\mathcal T$ are positive then for each $k \geq 1$,  the operator $T^Q_{k,t}$ is positive because of the approximating Riemann integral by the partial sums. Therefore, clearly for each $t\geq 0$, the operator $\mathcal{T}_{\lambda,Q}$ is positive.
Now, it is enough to prove that $\mathcal{T}_{\lambda,Q} $ is Markov for each $t\geq 0$. Take any element $x\in\ck$. We use the induction for each $k \geq 1$.  Since for every $t\geq 0$, the operator $T_t$ and $Q$ are Markov, then $f(T_{t-s} Q T_tx)=1$, ($t-s\geq 0$). Therefore,
\begin{eqnarray*}
\displaystyle
f(T^Q_1 (t) x) &=& f \bigg(\int_0^t T_{t-s}Q T_s x ds\bigg)\\
&=&  f \bigg(\lim_{n\to\infty} \sum_{k=0}^n T_{t - s_j} Q T_{s_j} x \Delta s\bigg) \\
&=&  \lim_{n\to\infty} \sum_{k=0}^n f(T_{t - s_j} Q T_{s_j} x) \Delta s \\
&=&  \lim_{n\to\infty} \sum_{k=0}^n  \Delta s \\
&=& \int^t_0ds=t \\
\end{eqnarray*}

Moreover, let us assume that by induction for each $t\geq 0$ and each $x\in\ck$, we have $\displaystyle f(T^Q_{k,t}x) = \frac{t^k}{k!}$.

Then, for $k+1$, one gets
\begin{eqnarray*}
\displaystyle
f(T^Q_{k+1,t} x) &=& f \bigg(\int_0^t T_{t-s} QT_s x ds\bigg)\\
&=&  f \bigg(\lim_{n\to\infty} \sum_{k=0}^n T_{t - s_j}Q T^Q_{k,s_j}x \Delta s\bigg) \\
&=&  \lim_{n\to\infty} \sum_{k=0}^n f(T_{t - s_j}Q T^Q_{k,s_j}x) \Delta s \\
&=&  \lim_{n\to\infty} \sum_{k=0}^n \frac{s^k_{j}}{k!}  \Delta s \\
&=&\int^t_0\frac{s^k}{k!}ds\\
&=& \frac{ t^{k+1}}{(k+1)!}\\
\end{eqnarray*}

Finally for $\mathcal{T}_{\lambda,Q}$, we find
\begin{eqnarray*}
\displaystyle
f(T^{\lambda, Q-I}_t x) &=& f (e^{-\lambda t} [T_t + \sum_{k=1}^{\infty} \lambda^k T^Q_{k,t} ] x)\\
&=&  e^{-\lambda t} f([T_t + \sum_{k=1}^{\infty} \lambda^k T^Q_{k,t}]x) \\
&=&  e^{-\lambda t} [f(T_tx)  + \sum_{k=1}^{\infty} \lambda^k f(T^Q_{k,t}x) ]\\
&=& e^{-\lambda t} \bigg(1 + \sum_{k=1}^{\infty} \frac{(\lambda t)^k }{k!}\bigg) = 1.
\end{eqnarray*}

If $\mathcal T \in \Xi_P(X)$, then from \eqref{TQ1} and \eqref{TQ2}, we infer that $\mathcal{T}_{\lambda,P}\in \Xi_P(X)$. This completes the proof.
\end{proof}

\begin{theorem}
Let $X$ be an abstract state space and $P$ be a Markov projection on $X$. Then the set $\Xi_P^{u} (X)$ is $\rho_1$ dense and open in $\Xi_P^{inv} (X)$.
\end{theorem}

\begin{proof} The openness follows from Lemma \ref{TPO}. Therefore, let us establish that $\Xi_P^{u} (X)$ is dense. Namely, given a semigroup $\mathcal T \in \Xi_P^{inv} (X)$ and $\varepsilon> 0$, we show that there exists $\mathcal S \in \Xi_P^{u} (X)$ such that $\rho_1 (\mathcal T, \mathcal S) < \varepsilon$.

Let $\lambda>0$ such that $2(1-e^{-\lambda})<\varepsilon$, and assume that $\mathcal S ={\mathcal T}^{\lambda, P-I}$ (see \eqref{TQ2}). By Lemma \ref{TQM}, we infer that $\mathcal S\in \Xi_P^{inv} (X)$.
Moreover, for every $x\in \ck$ and $k \in\mathbb N$, we obtain
\begin{eqnarray*}
\displaystyle
\| T_tx - T^{\lambda, P-I}_t x\| &=& \bigg\| T_tx - e^{-\lambda t} \big[T _t+ \sum_{k=1}^{\infty}\lambda^k T^P_{k,t} \big] x \bigg\| \\
&\leq& (1- e^{-\lambda t}) \| T_tx \| + e^{-\lambda t} \sum_{k=1}^{\infty} \lambda^k \|T^P_{k,t}x \|  \\
&=& 1- e^{-\lambda t}  + e^{-\lambda t} \sum_{k=1}^{\infty} \lambda^k  \frac{t^k}{k!} \\
&=& 2 (1- e^{-\lambda t}).
 \end{eqnarray*}

Hence
\begin{eqnarray*}
\rho_1 (\mathcal T, \mathcal S) &=&\sup_{t\in[0,1], x\in\ck} \|T_t x -  T^{\lambda, P-I}_t x \|\\[2mm]
&\leq& \sup_{t\in[0,1]} 2(1-e^{-\lambda t})\\
&=& 2(1-e^{-\lambda})<\varepsilon.
\end{eqnarray*}

Now, one needs to demonstrate that $\mathcal S$ is uniformly $P$-ergodic. For all $x\in N_P$ and arbitrary $t_0 > 0$, we have
$$
\displaystyle
T^P_{1,t_0} x = \int_0^{t_0} T_{t_0 - s} P T_s x ds = \int_0^{t_0} T_{t_0 - s}T_sP x ds = 0,
$$
and hence for all $k \geq 2$
$$
\displaystyle
T^P_{k,t_0} x = \int_0^{t_0} T_{t_0 - s}P T_{k,s}^P x ds = 0.
$$
Therefore
\begin{eqnarray*}
\| T^{\lambda, P-I}_{t_0} x \| &=&\| e^{-\lambda t_0} [T_{t_0} + \sum_{k=1}^{\infty} \lambda^k T^P_{k,t_0} ]x \| \\
&\leq & e^{-\lambda t_0}  \|T_{t_0}\|\|x\|\\
&  \leq&  e^{-\lambda t_0}\|x\|.
\end{eqnarray*}
Hence, $\d_P (T^{\lambda, P-I}_{t_0})\leq e^{-\lambda t_0} <1$, which by Theorem \ref{Dobb} implies $\mathcal S \in \Xi_P^{u}$.
This completes the proof.
\end{proof}

We denote
\begin{eqnarray*}
\Xi^{u}(X) := \{\mathcal T \in \Xi(X): \mathcal T \  \text{is uniformly ergodic}\}.
\end{eqnarray*}
then using the same argument we can prove the following result.
\begin{theorem}
Let $X$ be an abstract state space. Then the set $\Xi^{u} (X)$ is $\rho_1$ dense and open in $\Xi(X)$.
\end{theorem}

\begin{remark} The last theorem extends the results of the papers \cite{BK1,Kuna,LM92,R97} to more general spaces.
\end{remark}

\section{Uniform mean $P$-ergodicity of $C_0$-Markov semigroups}

Let $\mathcal T =(T_t)$ be a $C_0$-semigroup. Then the Cesaro averages is
$$
\displaystyle
{\mathcal A}_t (\mathcal T) = \frac{1}{t} \int_0^t T_s ds
$$
where the integral above is taken with respect to the strong operator topology.  The Cesaro averages is one of the well-known example for LR-nets. Uniform ergodicity of LR-nets of Markov operators on abstract state spaces in in terms of the Dobrushin ergodicity coefficient is studied in \cite{EM2018}. This section is devoted to study uniform mean ergodicity of $C_0$-semigroups by means of a generalized Dobrushin ergodicity coefficient.

\begin{definition} Let $P$ be a projection. A $C_0$-Markov semigroup $\mathcal T=(T_t)_{t\geq 0}$ is defined on $X$  is called
\begin{enumerate}
\item[(i)] {\it mean-$P$-ergodic} if
 for every $x\in X$
$$
\lim_{t\to\infty}\| \A x - Px \|=0;
$$
\item[(ii)] {\it uniform mean-$P$-ergodic} if
$$
\lim_{t\to\infty}\| \A  - P \|=0;
$$
\item[(iii)] \textit{ weakly-$P$-mean ergodic} if
$$
\lim_{t\to\infty} \d_P (\A) =0.
$$
\end{enumerate}
\end{definition}

%
%
%
%

 \begin{theorem}\label{UME}
Let $\mathcal T \in \Xi(X)$ and $\mathcal T$ be mean $Q$-ergodic. Then the following statements are equivalent:
\begin{itemize}
\item[(i)] $\mathcal T$ is uniformly $Q$-mean ergodic;

\item[(ii)] There exists  $t_0 \in\br_+$ such
that $\d_Q ({\mathcal A}_{t_0}(T) )<1$.
\end{itemize}
Moreover
\begin{equation}\label{ume}
\displaystyle
\|\A - Q\| \leq \frac{2 t_0}{1 - \d_Q  ({\mathcal A}_{t_0}(T) )} \frac{1}{t}.
\end{equation}
\end{theorem}

\begin{proof}
 (i) $\Rightarrow $(ii). Since $\mathcal T$ is uniformly $Q$-mean ergodic, then $\A$ is uniformly $P$-ergodic.  Then, using the
  argument of Theorem \ref{Dobb} there exists $t_0>0$ such that $ \d_Q ({\mathcal A}_{t_0}(T) ) < 1$.

 Let us prove (ii)$\Rightarrow$ (i).  Assume that there exists $t_0\in\br_+$ such that $\rho = \d_Q ({\mathcal A}_{t_0}(T))<1$. Then for an arbitrary $s > 0$,
\begin{eqnarray*}
\displaystyle
\left\| \A (I - T_s) \right\| &=& \left\|\frac{1}{t} \int_0^t T_u du - \frac{1}{t} \int_0^t T_u T_s du \right\| \\
&=& \left\|  \frac{1}{t} \int_0^t T_u du - \frac{1}{t} \int_0^{t} T_{u+s}  du \right\|   \ \ \ \ \text{use change of variable as } \ u+s = z \\
\\
&=& \left\|  \frac{1}{t} \int_0^t T_u du - \frac{1}{t} \int_s^{t+s} T_z  dz \right\| \\
&\leq&  \frac{1}{t} \int_0^s  \left\|T_u \right\| du + \frac{1}{t}
\int_t^{t+s} \left\| T_u \right\| du  \leq \frac{2s}{t}.
\end{eqnarray*}
Hence, for every $s\in\br_+$ one gets
\begin{eqnarray*}
\left\| \A (I- {\mathcal  A}_{s} (\mathcal T)) \right\|& =& \left\|
A_t(\mathcal T) \bigg(\frac{1}{s} \int_{0}^s
(I-T_u) du\bigg) \right\| \\[2mm]
&=& \left\| \frac{1}{s} \int_{0}^{s} A_t(\mathcal T) (I-T_u) du
\right\| \\[2mm]
&\leq & \frac{1}{s} \int_{0}^{s}\| A_t(\mathcal T) (I-T_u)\| du\\[2mm]
&\leq & \frac{1}{s} \int_{0}^{s}\frac{2u}{t}du\\[2mm]
&=&\frac{s}{t}
\end{eqnarray*}

By Theorem \ref{Dob}, we have
\begin{equation}\label{dd1}
\d_Q (\A (I- {\mathcal A}_{t_0} ) (\mathcal T))\leq \frac{t_0}{t}
\end{equation}

Due to the mean $Q$-ergodicity of $\ct$, we have $\ca_t Q=Q\ca_t$, hence again using Theorem \ref{Dob}(vi) one finds
\begin{eqnarray}\label{dd2}
\d_Q ({\mathcal A}_t(\mathcal T)(I- {\mathcal A}_{t_0}(\mathcal T))&\geq&  \d_Q ({\mathcal A}_t(\mathcal T))-\d_Q ({\mathcal A}_t(\mathcal T)
{\mathcal A}_{t_0}(\mathcal T))\nonumber\\[2mm]
&\geq&  \d_Q ({\mathcal A}_t(\mathcal T))-\d_Q ({\mathcal A}_t(\mathcal T))\d_Q( {\mathcal A}_{t_0}(\mathcal T))\nonumber\\[2mm]
&\geq&  (1-\rho)\d_Q \A,
\end{eqnarray}
which implies $\d_Q (\A ) \leq \frac{t_0}{t ( 1 - \rho)}$.
The weakly $Q$-mean ergodicity of  $\mathcal T$ implies its uniform $Q$-mean ergodicity.

Now we need to show \eqref{ume}. Noting that
$$
\displaystyle
\d_Q \A = \sup_{x\in N_Q} \frac{\|\A x\|}{\|x\|}.
$$
and $x\in N_Q$, there exists $y\in X$ such that $x = y -Qy$. Hence
$$
\d_Q \A \geq \sup_{y \in X} \frac{\|\A y - \A Q y \|}{\|y - Qy\|} = \sup_{y \in X} \frac{\|\A y -  Q y \|}{\|y - Qy\|}  \geq \frac{1}{2} \|\A - Q\|.
$$
This implies $\|\A - Q\| \leq 2 \d_Q \A \leq \frac{2 t_0}{t ( 1- \rho)}$.
\end{proof}

\begin{remark} We point out that in \cite{L} the uniform mean $P$-ergodicity and its relations with the generator of $(T_t)$ and its resolvent were investigated. Recently, in \cite{TKB} the uniform mean $P$-ergodicity and its equivalence to uniform Abel ergodicity have been studied. However, our result is essential in terms of the generalized Dobrushin ergodicity coefficient, and it gives a rate of
convergence of the Cesaro averages.
\end{remark}

%

\begin{proposition}
Let $\mathcal T \in \Xi_P (X)$ and let $P$ be a Markov projection.
 Then the followings are equivalent:
\begin{itemize}
\item[(i)] There exists $t_0$, and there exists $n_0$ such that $\delta_P ({\mathcal A}_{t_0}^{n_0} (\mathcal T) ) < 1$.
\item[(ii)] $\mathcal T$ is weakly-$P$-mean ergodic.
\end{itemize}
\end{proposition}

\begin{proof}
(ii) $\Rightarrow $(i) Since $\mathcal T$ is weakly-$P$-mean ergodic,
$\displaystyle \lim_{t\to\infty} \delta_P ({\mathcal A} _t(\mathcal T)) =0$. Then (i) is a direct conclusion.

(i) $\Rightarrow $(ii)  Assume that there exist $t_0 \in\mathbb R_+$ and $n_0 \in \mathbb N$ such that
$\rho := \delta_P ({\mathcal A}_{t_0}^{n_0} (\mathcal T) ) < 1$. Due to $T_t P = P T_t$, ${\mathcal A}_t (\mathcal T) P = P {\mathcal A}_t (\mathcal T) $, one finds
\begin{eqnarray*}
\left\| {\mathcal A}_t (\mathcal T)  ( I - {\mathcal A}_{t_0}^{n_0} (\mathcal T) ) \right\| &\leq &
\left\| {\mathcal A}_t (\mathcal T)  ( I - {\mathcal A}_{t_0} (\mathcal T))  (I + {\mathcal A}_{t_0} (\mathcal T) + {\mathcal A}_{t_0}^2 (\mathcal T) + \cdots +{\mathcal A}_{t_0}^{n_0 - 1} (\mathcal T) \right\| ) \\
&\leq& \left\| {\mathcal A}_t (\mathcal T)  ( I - {\mathcal A}_{t_0} (\mathcal T)) \right\| ( \left\| I \right\| +  \left\| {\mathcal A}_{t_0} (\mathcal T) \right\|  + \left\| {\mathcal A}_{t_0}^2 (\mathcal T) \right\|  \\
& &  + \cdots + \left\| {\mathcal A}_{t_0}^{n_0 - 1} (\mathcal T) \right\| )  \to 0.
\end{eqnarray*}
Therefore $\delta_P ( {\mathcal A}_t (\mathcal T)  ( I - {\mathcal A}_{t_0}^{n_0} (\mathcal T) ) $ converges to zero as
$t \to \infty$ and also $$| \delta_P( {\mathcal A}_t (\mathcal T) )  - \delta_P( {\mathcal A}_t (\mathcal T) {\mathcal A}_{t_0}^{n_0} (\mathcal T) )| \to 0  \  \  \text{as} \  \  t\to\infty$$.
By
\begin{eqnarray*}
\delta_P ({\mathcal A}_t (\mathcal T)  ( I - {\mathcal A}_{t_0}^{n_0} (\mathcal T))  &\geq&\delta_P( {\mathcal A}_t (\mathcal T) )  - \delta_P( {\mathcal A}_t (\mathcal T) {\mathcal A}_{t_0}^{n_0} (\mathcal T) ) \\
&\geq& \delta_P({\mathcal A}_t (\mathcal T))   (1 - \rho),
\end{eqnarray*}
it follows that $\displaystyle  \delta_P ({\mathcal A}_t (\mathcal T)  ) $ converges to zero.
\end{proof}

Historically, one of the most significant conditions for ergodicity is Doeblin's
Condition \cite{Lz,Num}: Let $(E,\cf)$ be a measurable space and $P(x,\cdot)$ be a Markov chain on $E$. There exist a probability measure $\mu$, an integer $m_0\geq 1$ and constants
$\epsilon< 1$ and $\kappa > 0$ such that for every $A\in \cf$ if $\mu(A)\geq\epsilon$ then
$$
\inf_{x\in E} P^{m_0}(x,A)\geq \kappa.
$$

Now, we are going to introduce a mean analogue of Doeblin's Condition.

Recall that for a Markov projection  $Q: X\to X$  we write $Q\leq P$ if $Q=QP=PQ$.

Let $(X, X_+, \mathcal K, f)$ be an abstract state space. Assume that $P$ is a Markov projection on $X$ and let $\mathcal T \in \Xi_P (X)$.
We say that $\ct$ satisfies \textit{condition $\frak{D}_P^{(m)}$}: if there exist constants
$\t\in(0,1]$,  $t_0>0$ and a Markov projection $Q$ with $Q\leq P$ such that for every $x\in \ck$  one finds $\f_{x}\in X_{+}$ with
$$
\sup\limits_{x}\|\f_{x}\|\leq \frac{\t}{4}
$$
such that
\begin{equation}\label{Dm}
\ca_{t_0}(\ct)x+\f_{x}\geq\t Qx.
\end{equation}

\begin{theorem}
Assume that $(X, X_+, \mathcal K, f)$ be a strong abstract state space. Let $P$ be a Markov projection on $X$ and let $\mathcal T \in \Xi_P (X)$. If $\mathcal T$ satisfies the condition $\mathcal D_P^{(m)}$, then
$\mathcal T$ is weakly $P$-mean ergodic.
\end{theorem}

\begin{proof}
By condition $\frak{D}_P^{(m)}$, there are $\t\in (0,1]$, $t_0>0$ and $Q\leq P$ such that for any two elements
$x,y\in\ck$ with $x-y\in N_P$, one finds
$\f_{x}, \f_{y}\in X_{+}$ with
\begin{eqnarray}\label{Dm2}
\sup\limits_{x}\|\f_{x}\|\leq \frac{\t}{4}
\end{eqnarray}
such that
\begin{equation}\label{Dm3}
\ca_{t_0}(\ct)x+\f_{x}\geq\t Qx, \ \ \ca_{t_0}(\ct)y+\f_{y}\geq\t Qy.
\end{equation}

By putting $\f_{xy}=\f_{x}+\f_{x}$, from \eqref{Dm3} we obtain
\begin{equation}\label{Dm31}
\ca_{t_0}(\ct)x+\f_{xy}\geq\t Qx, \ \ \ca_{t_0}(\ct)y+\f_{xy}\geq\t Qy.
\end{equation}

By the Markovianity of $\mathcal T$, and \eqref{Dm31}, \eqref{Dm2},  one gets
 \begin{eqnarray}\label{Dm32}
 \| {\mathcal A}_{t_0} (\mathcal T) x + \phi_{xy} - \tau Qx \| &=& f( {\mathcal A}_{t_0} (\mathcal T) x + \phi_{xy} - \tau Qx)\nonumber \\
 &=& 1 - \tau +f (\phi_{xy})\nonumber\\
 &\leq& 1- \frac{\tau}{2}.
 \end{eqnarray}

By the same argument, we have
 \begin{eqnarray}\label{Dm33}
 \| {\mathcal A}_{t_0} (\mathcal T) y + \phi_{xy} - \tau Qy\|\leq 1- \frac{\tau}{2}.
 \end{eqnarray}

Due to $P(x-y)=0$ and $Q\leq P$, one gets $Qx=Qy$. Therefore, from \eqref{Dm32},\eqref{Dm33}, we obtain
\begin{eqnarray*}
\|\ca_{t_0}(\ct)x-\ca_{t_0}(\ct)y\|&=&\|\ca_{t_0}(\ct)x+\f_{xy}-\t Qx-(\ca_{t_0}(\ct)y+\f_{xy}-\t Qy)\|\\[2mm]
&\leq & \|\ca_{t_0}(\ct)x+\f_{xy}-\t Qx\|+\|\ca_{t_0}(\ct)y+\f_{xy}-\t Qy\|\\[2mm]
&\leq & 2\bigg(1-\frac{\t}{2}\bigg)
\end{eqnarray*}
this by Remark \ref{SS1} implies
$$
\d_P(\ca_{t_0}(\ct))\leq 1-\frac{\t}{2}<1.
$$
Hence, by Theorem \ref{UME}, we arrive at the assertion.
\end{proof}

Let us consider a simple example of discrete semigroup for which $\frak{D}_P^{(m)}$ is satisfied.

\begin{example} Let $M=M_2(\mathbb{C})$ be the algebra of $2\times 2$ matrices. By
$\sigma_1,\s_2,\s_3$ we denote the Pauli matrices, i.e.
\begin{equation*}
\s_1=\left(
      \begin{array}{cc}
        0 & 1 \\
        1 & 0 \\
      \end{array}
    \right) \
\s_2=\left(
       \begin{array}{cc}
         0 & -i \\
         i & 0 \\
       \end{array}
     \right) \
\s_3=\left(
       \begin{array}{cc}
         1 & 0 \\
         0 & -1 \\
       \end{array}
     \right).
\end{equation*}

It is known that the set $\{\id,\sigma_1,\s_2,\s_3\}$ forms a
basis for $M_2(\mathbb{C})$. Every matrix
$x\in\textit{M}_2(\mathbb{C})$ can be written in this basis as
$x=w_0\id+\wb\cdot\s$ with $w_0\in\mathbb{C},
\wb=(w_1,w_2,w_3)\in\mathbb{C}^3,$ here by $\wb\cdot\s$ we mean
the following $\wb\cdot\s=w_1\s_1+w_2\s_2+w_3\s_3.$  We note that $x$ is positive iff $\|\wb\|\leq |w_0|$.

Let us consider a mapping
$\Phi_{\l,\m,\kappa}: M_2(\mathbb{C})\rightarrow M_2(\mathbb{C})$ given by
\begin{equation}\label{ks4}
\Phi_{\l,\m,\kappa}(w_0\id+\wb\cdot\s)=w_0\id+\l w_1\s_1+\m w_2\s_2+\kappa w_3\s_3
\end{equation}
which is clearly a Markov operator (see \cite{M}) if and only if
$\max\{|\l|,|\m|,|\kappa|\}\leq 1$.

Let us consider the following Markov operators $\Phi:=\Phi_{-1,0,1}$, $P=\Phi_{0,0,1}$. It is evident that $P$ is a projection with
$P\Phi=\Phi P=P$. One can see that
$$
\Phi^n(w_0\id+\wb\cdot\s)=w_0\id+(-1)^n w_1\s_1+w_3\s_3
$$
which yields that $\Phi$ is not uniformly ergodic.
However,
$$
A_n(\Phi)=\frac{1}{n}\sum_{k=1}^n\Phi^k=w_0\id+w_3\s_3-\frac{\chi_{\bn^{odd}}(n)}{n} w_1\s_1
$$
where $\bn^{odd}$ is the set of odd natural numbers. This implies that
$A_n(\Phi)\to P$, so $\Phi$ is uniformly mean $P$-ergodic.

Take any $0<\tau<1$. Then from
\begin{equation}\label{AAF}
A_{n_0}(\Phi)(x)\geq \t Px, \ \ x\geq 0
\end{equation}
we infer that
$$
(1-\t)w_0\id +(1-\t)w_3\s_3-\frac{\chi_{\bn^{odd}}(n_0)}{n_0} w_1\s_1\geq 0
$$
which means
$$
w_0\id +w_3\s_3-\frac{\chi_{\bn^{odd}}(n_0)}{(1-\t)n_0} w_1\s_1\geq 0.
$$
If $(1-\t)n_0\geq 1$, then
\begin{eqnarray*}
|w_3|^3+\bigg|\frac{\chi_{\bn^{odd}}(n_0)}{(1-\t)n_0}\bigg|^2|w_1|^2\leq |w_3|^3+|w_1|^2\leq |w_0|^2
\end{eqnarray*}
here we have used the positivity of $x$.

Hence, if $n_0>1/(1-\t)$, than \eqref{AAF} is satisfied. This yields that the fulfilment of $\frak{D}_P^{(m)}$.
\end{example}


\bibliographystyle{amsplain}

\end{document}